\documentclass[11 pt, reqno]{amsart}

\usepackage[letterpaper, hmarginratio=1:1]{geometry}
\usepackage[noadjust]{cite}
\usepackage{color}
\usepackage{bbm}

\theoremstyle{definition}
\newtheorem{ccounter}{ccounter}[section]
\newtheorem{theorem}[ccounter]{Theorem}

\newtheorem{lemma}[ccounter]{Lemma}

\newtheorem{definition}[ccounter]{Definition}
\newtheorem{remark}[ccounter]{Remark}

\newcommand\C{\mathbb{C}}
\newcommand\E{\mathbb{E}}
\newcommand\N{\mathbb{N}}

\newcommand\R{\mathbb{R}}
\newcommand\be{\begin{equation}}
\newcommand\ee{\end{equation}}

\def\Im{\operatorname{Im}}
\def\Re{\operatorname{Re}}
\def\one{{\mathbbm 1}}
\def\var{\operatorname{Var}}
\renewcommand{\tilde}{\widetilde}
\newcommand{\pf}{\operatorname{Pf}}
\newcommand{\iu}{\mathrm{i}}

\title[Central limit theorem for the  eigenvalues of Gaussian random matrices]{Central limit theorem for the complex eigenvalues of  Gaussian random matrices}
\author{Advay Goel}
\address{Massachusetts Institute of Technology, Cambridge, Massachusetts 02139}
\email{advayg@mit.edu}
\author{Patrick Lopatto}
\address{Division of Applied Mathematics\\Brown University\\Providence, Rhode Island 02906}
\email{Patrick\_Lopatto@brown.edu}
\author{Xiaoyu Xie}
\address{Division of Applied Mathematics\\Brown University\\Providence, Rhode Island 02906}
\email{Xiaoyu\_Xie@brown.edu}

\begin{document}

\begin{abstract} 
We establish a central limit theorem for the eigenvalue counting function of a matrix of real Gaussian random variables.
\end{abstract}

\maketitle



\section{Introduction}

\subsection{Main Result}
This note proves a central limit theorem (CLT) for the eigenvalue counting function of a matrix of real Gaussian random variables in regions of the complex plane.
While such a result is well known for matrices of complex Gaussians (see \cite[Section 3.1]{byun2022progress} for a survey), to the best of our knowledge, the analogous statement for 
real Gaussian matrices has not previously been addressed. 

We begin by defining the random matrix ensemble of interest in this work. 

\begin{definition}

For all $N\in \N$, let $G_N = (g_{ij} )_{1\le i,j \le N}$ be a random matrix whose entries are mutually independent Gaussian random variables with mean zero and variance one.  
We call $G_N$ the \emph{real Ginibre matrix} (GinOE) of dimension $N$. 
We also denote $W_N = N^{-1/2} G_N$. 
\end{definition}

In the limit as $N$ goes to infinity, it is known that the empirical spectral distribution of $W_N$ tends to the uniform measure on the  unit disk $\mathbb{D} = \{ z \in \mathbb{C} : |z| < 1 \}$ \cite{bordenave2012around}. We note that the eigenvalues of $W_N$ come in conjugate pairs, since $W_N$ is real; if $\lambda \in \mathbb{C}$ is an eigenvalue, then so is $\bar \lambda$. 
It is therefore natural when studying the fluctuations of the eigenvalues of $W_N$ to restrict attention to the upper half disk $\mathbb{D}^+ = \{ z \in \mathbb{C} : |z| < 1, \Im z > 0  \}$. We recall that a domain is defined as a non-empty connected open subset of $\mathbb{C}$.

\begin{definition}
We say that a domain $A$ is \emph{admissible} if   $\overline{A} \subset \mathbb{D}^+$. 
\end{definition}
This condition is slightly stronger than requiring $A\subset \mathbb{D}^+$, since it enforces a separation between $A$ and the boundary of $\mathbb{D}^+$.
We also recall that a domain is said to be Lipschitz if its  boundary is locally the graph of a Lipschitz continuous function; see \cite[Definition 12.9]{leoni2017first}. Given an admissible Lipschitz domain $A$, we let  $\ell(\partial A)$ denote the length of its boundary.

Denote the eigenvalues of $W_N$ by $\lambda_1, \dots, \lambda_N$, in an arbitrary order. Given an admissible domain $A$, we define $f_A: \mathbb{C} \rightarrow \R$ by $f_A(z) = \one_A(z)$, and define the ($N$-dependent) random variable 
\be\label{teststat}
X_A = \sum_{i=1}^N f_A(\lambda_i)  - \E\left[ \sum_{i=1}^N f_A(\lambda_i) \right].
\ee

The following theorem is our main result.
We let  $\mathcal N(0, c)$ denote a Gaussian random variable with mean zero and variance $c>0$. 
\begin{theorem}\label{t:main}
Let $A$ be an admissible Lipschitz domain. Then we have the weak convergence 
\be\label{mainlimit}
\lim_{N\rightarrow \infty}
\frac{  X_A}{N^{1/4}}
 = \mathcal N\bigg( 0, \frac{\ell(\partial A)}{2 \pi^{3/2} } \bigg).
\ee
\end{theorem}
The variance of $X_A$ is of order $N^{1/2}$, which is smaller than the variance of order $N$ seen in sums of independent random variables. This is due to the strong correlations between the eigenvalues of $W_N$  \cite{martin1980charge}. 
Further, the variance of the Gaussian in \eqref{mainlimit} is identical to the one in the analogous theorem for complex Gaussian matrices \cite[(3.9)]{byun2022progress}. 

\subsection{Background}

The analogue of Theorem~\ref{t:main} for a complex Ginibre matrix (GinUE) is known. 
It is a consequence of a theorem that provides a CLT for a broad class of determinantal point processes proved in  \cite[Section 2]{soshnikov2000gaussian} (see also \cite{costin1995gaussian}), together with the the explicit computation of the asymptotic variance in \cite[Corollary 1.2.1]{lin2023nonlocal}.\footnote{While \cite{costin1995gaussian} gives details only for certain Gaussian matrices, the authors note (in a remark attributed to H.\ Widom) that their method works in much greater generality, as later demonstrated in \cite{{soshnikov2000gaussian}}.} 
See \cite[Corollary 1.7]{charles2020entanglement} for an alternative proof in the case where $A$ has a smooth boundary. 
Further, a local CLT for the counting function 
of the GinUE eigenvalues was derived in \cite{{forrester2014local}}.

All of these works crucially rely on the fact that the eigenvalues of the GinUE form a determinantal point process. While this determinantal structure enables a precise analysis of many aspects of the GinUE, it is absent in the GinOE. Instead, the eigenvalues of the GinOE form a Pfaffian point process, and consequently they are more difficult to study \cite{byun2023progress}. 

Previous work on linear statistics of the GinOE has considered smooth test functions of the complex eigenvalues \cite{kopel2015linear,o2023partial}, 
differentiable functions of the real eigenvalues \cite{simm2017central,kopel2015linear, fitzgerald2021fluctuations}, general functions of the real eigenvalues \cite{fitzgerald2021fluctuations},
and the number of real eigenvalues \cite{fitzgerald2021fluctuations,simm2017central,forrester2023local,edelman1994many,edelman1997probability,kanzieper2005statistics}.
There have also been a few recent articles proving CLTs for linear statistics of matrices of general i.i.d.\  random variables when the test function has at least two derivatives \cite{cipolloni2022mesoscopic,cipolloni2021fluctuation,cipolloni2019central}. Proving a CLT for the eigenvalue counting function in this more general setting remains an open problem.

\subsection{Outline} In Section~\ref{s:preliminaryresults}, we collect several preliminary lemmas, and show that the Pfaffian correlations of the GinOE eigenvalues may be quantitatively approximated by determinantal correlations. In Section~\ref{s:mainproof}, we compute the variance and higher cumulants of $X_A$, and show that they match those of the desired Gaussian distribution, concluding the proof of Theorem~\ref{t:main}. Using the results of \cite{lin2023nonlocal}, it is straightforward to extend Theorem~\ref{t:main} to all domains with finite perimeter (so-called Caccioppoli sets) and certain domains with fractal boundaries. We briefly discuss this point in Remark~\ref{r:spikey}.

\subsection{Acknowledgments} The authors thank P.\ Bourgade for suggesting the problem, C.\ D.\ Sinclair for helpful comments on the references \cite{borodin2009ginibre, sinclair2009correlation}, and P.\ J.\ Forrester for comments on a preliminary draft. They are grateful to the referees for several useful suggestions.\\ 
P.\ L.\  was supported by the NSF postdoctoral
fellowship DMS-2202891. X.\ X.\ was
supported by NSF grant DMS-1954351.

\section{Preliminary Results}\label{s:preliminaryresults}

Set $\C^* = \C {\setminus} \R$. We recall that for all $k \in \mathbb{N}$, the complex--complex correlation functions $\rho_k^{(N)}\colon (\C^*)^k \rightarrow \R$ for $G_N$ are defined by the following property \cite[(5.1)]{borodin2009ginibre}. For every 
compactly supported, bounded Borel-measurable function $f\colon (\C^*)^k \rightarrow \R$, we have 
\begin{equation}
\int_{(\C^*)^k } f(z_1,\dots, z_k) 
\rho_k^{(N)} (z_1, \dots, z_k) \,  dz_1\dots dz_k
 = 
 \E 
 \left[
 \sum_{(i_1, \dots, i_k) \in \mathcal{I}_k } f(w_{i_1},\dots, w_{i_k})
 \right],\label{corrdef}
\end{equation}
where $\mathcal{I}_k\subset \{1, \dots , N\}^N$ is the set of pairwise distinct $k$-tuples of indices, $\{w_i\}_{i=1}^N$ are the eigenvalues of $G_N$, and we use $dz_i$ to denote the Lebesgue measure on $\mathbb{C}$.
 We typically write $\rho_k$ 
instead of $\rho_k^{(N)}$, since the value of $N$ will be clear from context. 
We also recall that if $M=(M_{ij})_{i,j=1}^{2n}$ is a $2n \times 2n$ skew-symmetric matrix, its Pfaffian is defined as
\begin{equation}\label{pfaffian}
    \pf(M)=\frac{1}{2^n n !} \sum_{\sigma \in S_{2 n}} \operatorname{sgn}(\sigma) \prod_{i=1}^n M_{\sigma(2 i-1), \sigma(2 i)},
\end{equation}
where $S_{2n}$ is the symmetric group of degree $2n$.

The following lemma, taken from \cite[Appendix B.3]{mays2012geometrical}, identifies the correlation functions $\rho_k$ explicitly.
\begin{lemma}
The $k$-point complex--complex correlation functions of the $N$-dimensional real Ginibre ensemble $G_N$ are given by
\be\label{pfcorr}
    \rho_k(z_1, \ldots, z_k)=\pf(K(z_i, z_j))_{1 \leq i, j \leq k},
\ee
where $(K(z_i, z_j))_{1 \leq i, j \leq k}$ is a $2k \times 2k$ matrix composed of the $2\times 2$ blocks
\begin{align*}
    K(z_i, z_j)=\begin{pmatrix}
D_{N}\left(z_i, z_j\right) & S_{N}\left(z_i, z_j\right) \\
-S_{N}\left(z_j, z_i\right) & I_{N}\left(z_i, z_j\right)
\end{pmatrix},
\end{align*}
and $D_N$, $I_N$, and $S_N$ are defined by
\begin{gather*}
S_{N}(z, w)= \frac{\mathrm i e^{-(1 / 2)(z-\bar{w})^2}}{\sqrt{2 \pi}}(\bar{w}-z) G(z, w) s_{N}(z \bar{w}), \\
D_{N}(z, w) =\frac{e^{-(1 / 2)(z-w)^2}}{\sqrt{2 \pi}}(w-z) G(z, w) s_{N}(z w), \\
I_{N}(z, w) =\frac{e^{-(1 / 2)(\bar{z}-\bar{w})^2}}{\sqrt{2 \pi}}(\bar{z}-\bar{w}) G(z, w) s_{N}(\bar{z} \bar{w}),
\end{gather*}
where $z,w \in \C^*$ and 
\begin{gather*}
G(z, w)=\sqrt{\operatorname{erfc}(\sqrt{2} \Im(z)) \operatorname{erfc}(\sqrt{2} \Im(w))},\qquad \operatorname{erfc}(x) = \frac{2}{\sqrt{\pi}} \int_x^\infty \exp( - t^2)\, dt,\\
s_{N}(z)=e^{-z} \sum_{j=0}^{N-1} \frac{z^j}{j !}.
\end{gather*}
\end{lemma}
\begin{remark}\label{evenN}
The functions $\rho_k$ were first determined explicitly in \cite{forrester2007eigenvalue}. The Pfaffian form in \eqref{pfcorr} was derived in the case of even $N$ in \cite{borodin2009ginibre}. Subsequently, a variety of methods have been used to recover this form for all $N$ \cite{forrester2009method,sinclair2009correlation,sommers2008general} (see also \cite[Section 4.6]{mays2012geometrical}).
\end{remark}
By a change of variable it is straightforward to see that the $k$-th correlation function for the complex eigenvalues of $W_N$ is $N^{k} \rho_k(\sqrt{N} z_1, \dots \sqrt{N} z_k)$.
The following lemma is useful for controlling these functions and is proved in Section~\ref{s:technical}.
We let $d_A = \inf \{ |z-w| : z \in A, w\in \partial  \mathbb{D}^+ \}$ denote the distance between $A$ and the boundary of $\mathbb{D}^+$, and use the standard ``big O'' notation $O(\cdot)$  for estimates that hold in the limit $N\rightarrow \infty$. 

\begin{lemma}\label{l:23}
Let $A$ be an admissible domain. Then there exists a constant $c(d_A)>0$ such that  
\begin{gather*}
\sup_{z,w \in A} D_N(\sqrt{N} z,\sqrt{N}w) = O ( e^{-cN}),
\qquad \sup_{z,w \in A} I_N(\sqrt{N}z,\sqrt{N}w) = O ( e^{-cN}),\\
\sup_{z,w \in A} S_N(\sqrt{N}z,\sqrt{N} w) = O(1),
\end{gather*}
where the implicit constants in the asymptotic notation depend only on $d_A$. 
\end{lemma}

We next state a useful lemma about Pfaffians, proved in \cite[Appendix B]{gebert2019pure}.\footnote{The statement has appeared earlier in the literature, for example in \cite{forrester2008skew}.}
\begin{lemma}\label{l:24}
Let $M=(M_{ij})_{i,j=1}^{2n}$ be a skew-symmetric $2n \times 2n$ matrix such that $M_{ij} = 0$ when $i \equiv j \bmod 2$. Let $\tilde M = (\tilde M)_{i,j=1}^n $ be the $n\times n$ matrix formed by setting $\tilde M_{ij} = M_{2i-1, 2j}$. 
Then $\pf(M) = \det (\tilde M)$.
\end{lemma}

Finally, we require the following integral formula from \cite[Corollary 3.1.4]{lin2023nonlocal}.
\begin{lemma}\label{l:lin}
Let $J:\C \rightarrow \R$ be radially symmetric (meaning $J(z) = J(|z|)$) and nonnegative. Suppose further that $\int_{\C} J(z) \cdot |z| \, dz = 1$. Then for any admissible Lipschitz region $A$, 
\begin{equation*}
\lim_{N\rightarrow \infty} N^{3/2}
\int_A \int_{A^c} J\big( \sqrt{N}( z -w ) \big)\, dz \, dw = \frac{4}{\pi} \cdot \ell(\partial A).
\end{equation*}
\end{lemma}

\section{Proof of Theorem~\ref{t:main}}\label{s:mainproof}

\subsection{Variance Calculation}

The following lemma follows from results proved in \cite{kopel2015linear}. We sketch the proof for completeness.

\begin{lemma}\label{kopel}
For any admissible domain $A$, 
\begin{equation}\label{kopeleq}
\var[ X_A ] 
=  
\frac{N}{\pi} \operatorname{area}(A)
 - \frac{   N^2}{  \pi^2}
\int_{A} 
\int_{A}
  \exp(
 -N|z-w|^2
 )\, dz\, dw
+ O(N^{-1}),
\end{equation}
where the implicit constant in the asymptotic notation depends only on $d_A$.
\end{lemma}
\begin{proof}
From the definition \eqref{teststat} of $X_A$, we compute 
\begin{equation*}
\var[X_A]=
\E\left[\sum_{i=1}^N f_A(\lambda_i) \right]
+ 
\E\left[\sum_{i \neq j} f_A(\lambda_i)f_A(\lambda_j)
\right]
-
\E\left[ \sum_{i=1}^N f_A(\lambda_i) \right]^2.
\end{equation*}
Writing this expression in terms of correlation functions using \eqref{corrdef} and \eqref{pfcorr}, we obtain 
\begin{align}
\var[X_A] =& 
N \int_A S_{N}(\sqrt{N} z, \sqrt{N} z)\, dz 
- N^2 \int_{A^2} S_N(\sqrt{N} z, \sqrt{N} w)^2 \, dz \, dw\notag \\
& 
- N^2 \int_{A^2} D_N(\sqrt{N} z, \sqrt{N} w)I_N(\sqrt{N} z, \sqrt{N} w) \, dz \, dw.\label{inserthere}
\end{align}
The last term in \eqref{inserthere} vanishes exponentially, by Lemma~\ref{l:23}. The first term is computed in 
\cite[Lemma 7]{kopel2015linear} and equals
\begin{equation}\label{insert1}
\frac{N}{\pi} \operatorname{area}(A)
- \frac{1}{4 \pi} \int_A \frac{dz}{\Im(z)^2} + O(N^{-1}).
\end{equation}
The second term is computed in the proof of \cite[Lemma 9]{kopel2015linear} and equals
\begin{equation}\label{insert2}
- \frac{   N^2}{  \pi^2}
\int_{A} 
\int_{A}
  \exp(
 -N|z-w|^2
 )\, dz\, dw
+ \frac{1}{4 \pi} \int_A \frac{dz}{\Im(z)^2}  + O(N^{-1}).
\end{equation}
Inserting \eqref{insert1} and \eqref{insert2} into \eqref{inserthere} completes the proof.\footnote{While the main results of \cite{kopel2015linear} require the test function to be smooth, these calculations do not. Further, they were given for even $N$ in \cite{kopel2015linear} using the statement of \eqref{pfcorr} for even $N$ in \cite{borodin2009ginibre}. Their extension to odd $N$ requires only notational changes, given that \eqref{pfcorr} is now known for all $N$.}
We observe that the asymptotic bounds in the proofs of the cited lemmas rely only on Lemma~\ref{l:23} and the estimates Lemma~\ref{sNasymptotic} and \eqref{erf} stated below, whose error terms depend on $A$ only through $d_A$. This justifies the claim that the implicit constant in \eqref{kopeleq} depends only on $d_A$, even though this dependence was not made explicit in \cite{kopel2015linear}.
\end{proof}

\begin{lemma}\label{l:variance}
For any admissible Lipschitz domain $A$, 
\begin{equation*}
\lim_{N\rightarrow \infty}
\frac{\var[ X_A ]}{N^{1/2}} = \frac{1}{ 2 \pi^{3/2}} \cdot \ell( \partial A).
\end{equation*}
\end{lemma}
\begin{proof}
We write
\begin{align}
\int_{A} 
\int_{A}
  \exp(
 -N|z-w|^2
 )\, dz\, dw
 =&
 \int_{A} 
\int_{\mathbb{C}}
  \exp(
 -N|z-w|^2
 )\, dz\, dw \notag \\
 &-
  \int_{A} 
\int_{A^c}
  \exp(
 -N|z-w|^2
 )\, dz\, dw\label{p2}
\end{align}
By a change of variable and the Gaussian integral formula $\int_{\R} e^{-x^2} = \sqrt{ \pi}$, we have 
\begin{equation}\label{p3}
\int_{\mathbb{C}}
  \exp(
 -N|z-w|^2
 )\, dz
 = \int_{\mathbb{C}}
  \exp(
 -N|z|^2
 )\, dz = \frac{\pi}{N}.
\end{equation}
Combining \eqref{kopeleq}, \eqref{p2}, and \eqref{p3}, we obtain 
\begin{equation}\label{p4}
\var[ X_A ]  =
\frac{ N^2}{  \pi^2}
\int_{A} 
\int_{A^c}
  \exp(
 -N|z-w|^2
 )\, dz\, dw+O(N^{-1}).
\end{equation}
By Lemma~\ref{l:lin} applied  to the  radially-symmetric kernel function $J \colon \R^2 \rightarrow \R$ given by $J(r) =  2 \pi^{-3/2} \exp(-2 r^2)$, we find 
\begin{equation}\label{p5}
 \lim_{N\rightarrow \infty}
 N^{3/2}
\int_{A} 
\int_{A^c}
  \exp(
 -N|z-w|^2
 )\, dz\, dw
=  \frac{\sqrt{\pi}}{2}
\cdot \ell( \partial A).
\end{equation}
We conclude by combining \eqref{p4} and \eqref{p5}.
\end{proof}

\subsection{Higher Cumulants}

We recall that given a random variable $X$, its cumulants $\{\kappa_n\}_{n=1}^\infty$ are defined by 
\begin{equation*}
\log \E[ e^{\iu t X}]
= \sum_{n=1}^\infty \kappa_n \frac{(\iu t)^n}{n!},
\end{equation*}
and that for every $n\in \mathbb{N}$, there exists a degree $n$ polynomial $L_n$ (independent of the choice of $X$) such that $\kappa_n = L_n\big(  \E[X], \E[X^2], \dots, \E[X^n]\big)$. 

Let $Y$ be a point process on a subset $\mathcal D \subset \mathbb{C}$  with $N$ particles $\{y_i\}_{i=1}^N$ and correlation functions $\tau_k\colon \mathcal D^k \rightarrow \R$ such that
\begin{equation}
\int_{\mathcal D^k } f(z_1,\dots, z_k) 
\tau_k (z_1, \dots, z_k) \,  dz_1\dots dz_k
 = 
 \E 
 \left[
 \sum_{(i_1, \dots, i_k) \in \mathcal{I}_k } f(y_{i_1},\dots, y_{i_k})
 \right]\label{corrdef2}
\end{equation}
for all $k\in \mathbb{N}$ and all compactly supported, bounded Borel-measurable functions $f\colon \mathcal D^k \rightarrow \R$. For every domain $A \subset \mathbb{C}$, let $N_{A} = \sum_{i=1}^N \one_{A}(y_i)$ denote the counting function for $A$. 
We insert the test function $f(z_1,\dots,z_k) = \one_{A}(z_1)\cdots \one_{A}(z_k)$ into \eqref{corrdef2}, and note that the number of elements $(i_1, \dots ,i_k) \in \mathcal I_k$ such that $f(y_i, \dots, y_{i_k}) = 1$ is equal to $N_{A} (N_{A} -1) \cdots (N_{A} - k + 1)$, since there are $N_A$ choices for $i_1$, and then $N_A - 1$ choices remaining for $i_2$, and so on until $i_k$. This implies the well-known identity
\begin{equation}\label{descending}
\E\big[ N_{A} (N_{A} -1) \cdots (N_{A} - k + 1) \big] =  \int_{A^k} \tau_k( z_1, \dots ,  z_k) \, dz_1 \dots d z_k.
\end{equation}

Let $J_k$ denote the integral on the right-hand side of \eqref{descending}. Then \eqref{descending} implies that for all $n\in \mathbb{N}$, the moment $\E[N_{A}^n]$ is equal to a linear combination of the terms $J_1, \dots, J_n$, with universal coefficients (independent of $Y$). Recalling the definition of the polynomial $L_n$, we conclude that for every $n \in \mathbb{N}$, there exists a universal polynomial $H_n$  such that 
\begin{equation}\label{Hn}
\kappa_{n}(N_A) = H_n(J_1, \dots, J_n).
\end{equation}

The following lemma is a consequence of Lemma~\ref{l:23} and Lemma~\ref{l:24}. 
Let $A$ be an admissible domain, and for all $k\in \mathbb{N}$, set $Q^{(k)}(z_1,\dots , z_k) =(S_N(z_i, z_j))_{1 \leq i, j \leq k}$. We define 
\begin{equation}\label{Tk}
T_k = N^k\int_{A^k} \det Q^{(k)}(\sqrt{N} z_1,\dots , \sqrt{N} z_k) \, dz_1 \dots d z_k.
\end{equation}
\begin{lemma}\label{closecumulants}
For any admissible domain $A$, there exists a constant $c(d_A)>0$ such that $\kappa_n(X_A) = H_n(T_1,\dots, T_n) + O(e^{-cN})$ for all $n\in \mathbb{N}$. The implicit constant  depends only on $n$ and $d_A$. 
\end{lemma}
\begin{proof}
We begin by computing $\rho_k(\sqrt{N} z_1, \dots , \sqrt{N} z_k)$ using the definition of $\rho_k$ in \eqref{pfcorr} and the definition of a Pfaffian in \eqref{pfaffian}. By Lemma~\ref{l:23}, all terms in the defining sum \eqref{pfaffian} containing a factor of $D_N$ or $I_N$ are exponentially small. We conclude that  
\begin{equation}\label{edit1}
\sup_{z_1,\dots, z_k \in \Omega}
\big|
N^k \rho_k(\sqrt{N} z_1, \dots, \sqrt{N} z_k)
-
N^k \pf\big(\tilde K(\sqrt{N} z_i, \sqrt{N} z_j)\big)_{1 \leq i, j \leq k}
\big|
\le c^{-1} e^{-cN},
\end{equation}
where
where $(\tilde K(z_i, z_j))_{1 \leq i, j \leq k}$ is a $2k \times 2k$ matrix composed of the $2\times 2$ blocks
\begin{align*}
    \tilde K(z_i, z_j)=\begin{pmatrix}
0  & S_{N}\left(z_i, z_j\right) \\
-S_{N}\left(z_j, z_i\right) & 0 
\end{pmatrix}.
\end{align*}
Lemma~\ref{l:24} implies 
\begin{equation}\label{edit2}
\pf\big(\tilde K(\sqrt{N} z_i, \sqrt{N} z_j)\big)_{1 \leq i, j \leq k}
= \det Q^{(k)}(z_1, \dots, z_k)
.
\end{equation}
Combining \eqref{descending}, \eqref{edit1}, \eqref{edit2}, and the definition of $T_k$ in \eqref{Tk}, we find 
\begin{align*}
T_k &= N^k \int_{A^k} \rho_k(\sqrt{N} z_1, \dots, \sqrt{N} z_k)
\, dz_1 \dots d z_k
 + O(e^{-cN}) \\
 &=\E\big[ X_{A} (X_{A} -1) \cdots (X_{A} - k + 1) \big]  + O(e^{-cN})
,
\end{align*}
since $N^k \rho_k(\sqrt{N} z_1, \dots, \sqrt{N} z_k)$ is the $k$-th correlation function for the complex eigenvalues of $W_N$. The conclusion follows after recalling the definition of $H_n$ from \eqref{Hn} and using the trivial inequality $|T_k| \le 2 N^k$. 
\end{proof}
Lemma~\ref{closecumulants} motivates the next definition.
\begin{definition} 
We define the \emph{pseudo-cumulants} of $X_A$ by $\tilde \kappa_n = H_n(T_1, \dots, T_n)$ for all $n\in \mathbb{N}$.
\end{definition}
The following cumulant identity is known for determinantal processes \cite[(2.6)]{soshnikov2000central}. The proof in \cite{soshnikov2000central} works for the pseudo-cumulants without modification, since they are defined in terms of a determinantal kernel.\footnote{They are precisely the cumulants of the determinantal point process defined by the kernel $S_N(\sqrt N z, \sqrt N w)$, if such a process exists. We do not address the question of existence here, since this claim is not needed.}
\begin{lemma}
For all $n\in \mathbb{N}$, we  have
\begin{gather*}
\tilde \kappa_n = 
\sum_{m=1}^n 
\frac{(-1)^{m-1}}{m}
\sum_{\substack{n_1 + \cdots + n_m = n\\
n_1, \dots, n_m >0}} \frac{n!}{n_1!\cdots n_m!}\cdot  R_m,\\
R_m = N^m  \int_{A^m} \prod_{i=1}^m S_N(\sqrt{N} z_i, \sqrt{N} z_{i+1})\, dz_i,
\end{gather*}
with the convention that $z_{m+1} = z_1$. 
\end{lemma}

The next lemma follows from the previous one by induction; see \cite[Lemma 1]{soshnikov2000gaussian} for the statement in the case of determinantal processes.
\begin{lemma}\label{l:recursion} For all $n \in \mathbb{N}$, there exist constants $(\alpha_{nj})_{j=2}^{n-1}$ (independent of $A$ and $N$) such that  
\begin{equation*}
\tilde \kappa_n = (-1)^{n}(n-1)! (  R_1 -   R_n) + \sum_{j=2}^{n-1} \alpha_{n j} \tilde \kappa_{j}.
\end{equation*}
\end{lemma}
In light of the previous lemma, we now aim to calculate the terms $R_n$. 
\begin{lemma}\label{l:Rk}
Fix $\delta \in (0,1/2)$. For $k \ge 2 $, we have 
\begin{equation*}
R_1 = \frac{N}{\pi}\operatorname{area}(A) + O(1), \quad 
R_k = \frac{N}{\pi}\operatorname{area}(A)  + O(N^{1/2 + \delta}) ,
\end{equation*}
where the implicit constant in the asymptotic notation depends only on $d_A$, $k$, and $\delta$.
\end{lemma} 
To prepare for the proof, we recall the standard error function asymptotic
\begin{equation}\label{erf}
\operatorname{erfc}(x) = \frac{e^{-x^2}}{\sqrt{\pi} x} 
\left( 1 + O(x^{-2}) \right). 
\end{equation}
\begin{proof}[Proof of Lemma~\ref{l:Rk}]
The case $k=1$ is \cite[Lemma 7]{kopel2015linear}, so we suppose that $k \ge 2$. 
Then by \eqref{erf} and Lemma~\ref{sNasymptotic}, for all $z,w \in A$, we have the asymptotic expansion
\be\label{SNasymptotic}
S_N ( \sqrt{N} z , \sqrt{N} w) = 
U(z,w) \left(
1 -
\frac{e^{-2(1-z\bar w)}}{\sqrt{2\pi N} (1 - z \bar w) }
e^{N(1-z\bar w)}
(z\bar w)^{N}\right)
\left( 1 + O(N^{-1}) \right)
\ee
where
\begin{equation*}
U(z,w) = \frac{\iu e^{(-N/2)(z - \bar w)^2 -N (\Im(z)^2 + \Im(w)^2)} }{2\pi \sqrt{  \Im(z) \Im(w)}} (\bar w - z).
\end{equation*}
We claim that the leading order term in $R_k$ is $N^k \int_{A^k} \prod_{i=1}^k U(z_i , z_{i+1})\, dz_i$. To show this, we begin by illustrating how to bound one of the other terms in $R_k$ coming from \eqref{SNasymptotic}. We note that there exists a constant $C(d_A,k) > 0$ such that  
\begin{align}
&N^k \left| \int_{A^k} \prod_{i=1}^k U(z_i, z_{i+1})
\frac{e^{-2(1-z_i \bar z_{i+1} )}}{\sqrt{2N} \pi z_i \bar z_{i+1} }
e^{N(1-z_i\bar z_{i+1} )}
(z_i\bar z_{i+1} )^{N}  \, dz_i \right|\notag \\
&\le 
C N^k \int_{A^k} \prod_{i=1}^k
e^{(-N/2)\Re((z_i - \bar z_{i+1} )^2) -N (\Im(z_i)^2 + \Im(z_{i+1})^2)}
e^{N ( 1 - \Re(z_i \bar z_{i+1})  + \ln |z_i \bar z_{i+1} | )  }\, dz_i\notag \\
& = C N^k \int_{A^k} \prod_{i=1}^k e^{ (-N/2) ( |z_i|^2 + |z_{i+1}|^2 -2  - \ln|z_i|^2 - \ln |z_{i+1}|^2 ) } \, dz_i .\label{prev1}
\end{align}
We now observe that the integral in \eqref{prev1} decays exponentially in $z$, since $|z|^2 - 1 - \ln |z|^2 $ is positive and bounded away from zero for $z \in A$ (since $A$ is admissible). The other error terms can be treated similarly; each has an integrand that decays exponentially.

Introducing the notation 
\begin{equation*}
g(z,w) =
\Re(z) \Im(w) - \Re(w) \Im(z),
\end{equation*}
using \eqref{SNasymptotic}, and bounding the error terms as indicated in \eqref{prev1}, we obtain (after observing some cancellation in the exponent) that 
\begin{align}\label{hence}
R_k = \big( 1 + O(N^{-1}) \big)
N^{k} \int_{A^k}
&\exp
\left(
- \frac{N}{2}\sum_{i=1}^k |z_i - z_{i+1}|^2  + \iu N 
\sum_{i=1}^N g(z_i, z_{i+1})
\right)\\ &\times 
\prod_{i=1}^k
\frac{\iu(\bar z_{i+1} - z_i) }{ 2 \pi \Im(z_i)}
\, dz_i  +  O( e^{-cN} ), \notag
\end{align}
for some constant $c(d_A, k)>0$. 
We now decompose 
\begin{equation*}
\prod_{i=1}^k(\bar z_{i+1} - z_i) = 
\prod_{i=1}^k ( z_{i+1} - z_i - 2 \iu \Im (z_{i+1} )) = (-2i)^k\prod_{i=1}^k   \Im(z_i) 
 + \epsilon(z_1, \dots, z_k),
\end{equation*}
where $\epsilon(z_1,\dots, z_k)$ is the sum of terms containing at least one copy of $(z_i - z_{i+1})$. We  claim that all integrals arising from $\epsilon(z_1,\dots, z_k)$ are negligible. The following computation demonstrates this for terms containing exactly one copy of $(z_i - z_{i+1})$; the other terms are bounded similarly (and are lower order). We have 
\begin{align*}
N^{k}& \left| \int_{A^k}
\exp
\left(
- \frac{N}{2}\sum_{i=1}^k |z_i - z_{i+1}|^2  + \iu N 
\sum_{i=1}^N g(z_i, z_{i+1})
\right) (z_1 - z_2) \, dz_1 \dots dz_k
\right|\\
&\le N^k 
\int_{A^k}
\exp
\left(
- \frac{N}{2}\sum_{i=1}^k |z_i - z_{i+1}|^2 
\right) |z_1 - z_2| \, dz_1 \dots dz_k\\
&\le 
N^{k} \int_{A^k \cap \{  |z_1 - z_2| \le N^{-1/2 + \delta} \}}
\exp
\left(
- \frac{N}{2}\sum_{i=1}^k |z_i - z_{i+1}|^2 
\right) |z_1 - z_2| \, dz_1 \dots dz_k + O(e^{-cN}),
\end{align*}
due to the exponential decay of the integrand on the set $A^k \cap \{  |z_1 - z_2| > N^{-1/2 + \delta}\}$. 
We have 
\begin{align*}
&N^{k} \int_{A^k \cap \{  |z_1 - z_2| \le N^{-1/2 + \delta} \}}
\exp
\left(
- \frac{N}{2}\sum_{i=1}^k |z_i - z_{i+1}|^2 
\right) |z_1 - z_2| \, dz_1 \dots dz_k \\
& \le N^{ k -1/2 + \delta} 
\int_{ \C^{k-1} \times A}
\exp
\left(
- \frac{N}{2}\sum_{i=1}^{k-1} |z_i - z_{i+1}|^2 
\right) \, dz_1 \dots dz_k = O(N^{1/2 + \delta}),
\end{align*}
where the last inequality follows by directly evaluating the integrals in the variables $z_1$ through $z_{k-1}$, then using the fact that $\operatorname{area}(A) \le 2$.

After bounding these lower-order terms, \eqref{hence} becomes 
\begin{equation*}
R_k = \pi^{-k} N^{k}
\int_{A^k}
\exp
\left(
- \frac{N}{2}\sum_{i=1}^k |z_i - z_{i+1}|^2  + \iu N 
\sum_{i=1}^k g(z_i, z_{i+1})
\right)\, dz_1 \dots dz_k + O(N^{1/2 + \delta}).
\end{equation*}
We write $R_k$ as
\be\label{Rkfinal}
R_k =  \pi^{-k}N^k  \left( I_0 -  \sum_{j=1}^{k-1} I_j \right)  +  O(N^{1/2 + \delta}) ,
\ee
where 
\begin{equation*}
I_0 = \int_{A} \int_{\C^{k-1}}
\exp
\left(
- \frac{N}{2}\sum_{i=1}^k |z_i - z_{i+1}|^2  + \iu  N
\sum_{i=1}^k g(z_i, z_{i+1})
\right)\, dz_1 \dots dz_k,
\end{equation*}
and $I_j$ for $j\ge1$ is defined similarly to $I_0$, with the integral over $A \times \C^{k-1}$ replaced by one over $A^j \times A^c \times \C^{k-j-1}$. $I_0$ is the leading-order term, and may be computed explicitly.
After the change variables by $z_i \mapsto z_i +z_k$ for $i < k$, the variable $z_k$ disappears from the exponent and may be integrated directly. After some simplification, we obtain
\begin{equation*}
I_0 = 
\operatorname{area}(A) 
\int_{\C^{k-1}}
\exp\left( - N\sum_{i=1}^{k-1} |z_i|^2 + N\sum_{i=1}^{k-2} \overline{z}_i z_{i+1}
\right) dz_1\ldots dz_{k-1}
\end{equation*}
Changing variables to polar coordinates by setting $z_i = r_i e^{\iu \theta_i}$, and using the identity
\begin{equation*}
\int_0^{2\pi} \exp(\alpha e^{\iu \theta})\, d\theta = \oint \exp(\alpha z) \frac{dz}{\iu z} = 2 \pi,
\end{equation*}
valid for any $\alpha \in \mathbb{C}$, to integrate out the $\theta_i$ variables, we obtain\footnote{We learned of this integration method from \cite{elia2003integration}, which derives a general formula for integrals of exponentials of complex quadratic forms.}
\be\label{I0result}
I_0 = \pi^{k-1}N^{1-k} \operatorname{area}(A).
\ee


Next, we note that for every $j$ such that $1 \le j \le k-1$, we have 
\begin{equation*}
| I_{j} | \le 
\int_{A \times A^c \times \C^{k-2} } 
\exp
\left(
- \frac{N}{2}\sum_{i=1}^k |z_i - z_{i+1}|^2 \right)\, dz_1 \dots dz_k .
\end{equation*}
Recalling \eqref{p3}, have  
\begin{align*}
\int_{\C} \exp\left(
-\frac{N}{2} |z_{k-1} - z_k|^2 - \frac{N}{2} |z_k  - z_{1} |^2 
\right) \, dz_k 
& \le 
\int_{\C} \exp\left( - \frac{N}{2} |z_k  - z_{1} |^2 
\right) \, dz_k = \frac{\pi}{N}.
\end{align*}
The variables $z_{k-1},\dots, z_3$ can then be integrated directly using \eqref{p3}. By Lemma~\ref{l:lin}, 
\begin{equation*}
\int_{A \times A^c} \exp\left( - \frac{N}{2} |z_1  - z_2 |^2 \right)  \, dz_1 \, dz_2  = O( N^{-3/2}).
\end{equation*}
We conclude that  for $j\ge 1$,
\be\label{Ij}
I_{j} = O(N^{-k+1/2})
\ee
Inserting \eqref{I0result} and \eqref{Ij} into \eqref{Rkfinal} completes the proof.
\end{proof}

\subsection{Conclusion}
\begin{proof}[Proof of Theorem~\ref{t:main}]
By Lemma~\ref{closecumulants}, for every $n\in \mathbb{N}$ the cumulant $\kappa_n(X_A)$ is equal to the pseudo-cumulant $\tilde \kappa_n$ plus an exponentially small error term. Then by Lemma~\ref{l:variance}, Lemma~\ref{l:recursion}, Lemma~\ref{l:Rk}, and induction, we have for every $n \ge 3$ that 
\begin{equation*}
\lim_{N\rightarrow \infty} \kappa_2( N^{-1/4} X_A)  = \frac{\ell(\partial A)}{2 \pi^{3/2}},
\qquad
\lim_{N\rightarrow \infty} \kappa_n( N^{-1/4} X_A)  = 0.
\end{equation*}
We conclude that the limiting cumulants of $N^{-1/4} X_A$ are the same as the cumulants of a Gaussian random variable with variance $ 2^{-1} \pi^{-3/2} \ell(\partial A)$. Since the cumulants of a random variable determine its moments, the limiting moments also match those of this Gaussian. Because the Gaussian distribution is uniquely determined by its moments \cite[Theorem 30.1]{billingsley2017probability}, this implies the desired weak convergence
\cite[Theorem 30.2]{billingsley2017probability}.
\end{proof}
\begin{remark}\label{r:spikey}
The Lipschitz hypothesis in Theorem~\ref{t:main} was used only to compute the variance in Lemma~\ref{l:variance}. To relax this hypothesis, one only needs to compute the integral \eqref{kopeleq} for more general domains (with rougher boundaries). This can be done for  Caccioppoli sets using \cite[Corollary 3.1.4]{lin2023nonlocal} and for the Koch snowflake using \cite[Theorem 3.3.2]{lin2023nonlocal}. We note that the proof technique for the latter result is applicable to many other domains with self-similar boundaries.
\end{remark}
\section{Technical Estimates}\label{s:technical}

The following estimate improves \cite[Lemma 9.2]{borodin2009ginibre} by establishing a quantitative error term. It is implicit in \cite[Remark 3.4]{kriecherbauer2008locating}; we provide a short proof here for completeness. 
\begin{lemma}\label{sNasymptotic}
Let $A$ be an admissible domain. Define $\tilde d_A= \inf \{|z-1|:z\in A\}$.
Then there exists a constant $C(\tilde d_A)>0$ such that for all $z\in A$,
\begin{equation*}
s_{N}(Nz)
=
1 -
\frac{1}{\sqrt{2 \pi N}}
\frac{( z e^{1-z})^N}{1-z}
\big( 1 + R(z;N) \big), \qquad  \big|R(z;N)\big|\leq C N^{-1}.
\end{equation*}
\end{lemma}
\begin{proof}
By repeated integration by parts, we have 
\begin{equation}\label{expand1}
    s_N(Nz)=1-\frac{1}{(N-1)!}\int_0^{Nz}\zeta^{N-1}e^{-\zeta}d\zeta
    =1-\frac{N^N}{(N-1)!}\int_0^z\zeta^{N-1}e^{-N\zeta}d\zeta.
\end{equation}
The integral $\int_0^z$ can be taken along any curve connecting $0$ and $z$ since the integrand is analytic. 
Inserting Stirling's formula \cite{marsaglia1990new} 
\begin{equation}\label{stirling}
    N!=\sqrt{2\pi N}\left(\frac{N}{e}\right)^N\tau_N,\qquad \tau_N=1+\frac{1}{12N}+O(N^{-2})
\end{equation}
into \eqref{expand1}, we have
\begin{equation}\label{expand2}
    e^{-Nz}s_N(Nz)=1-\frac{1}{\tau_N}\sqrt{\frac{N}{2\pi}}\int_0^z\zeta^{N-1}e^{N(1-\zeta)}d\zeta.
\end{equation}

Consider the map $\varphi\colon\zeta\mapsto \zeta e^{1-\zeta}$ on $\C$. We recall that the function $W(z)$ solving the equation $-e\varphi(-W(z))=z$ is known as the Lambert $W$ function. Standard facts about this function imply that there is a multi-valued inverse of $\varphi$ with a (single-valued) principal branch defined on $\C \setminus [1,\infty)$ (see \cite[Section~4]{corless1996lambert}). It is given by $\psi(z) = - W(-z/e)$. 
Applying the change of variable $\zeta= \psi((1-t)ze^{1-z})$ to \eqref{expand2}, we have
\begin{equation}\label{expand3}
    \begin{aligned}
        s_N(Nz)=&1-\frac{1}{\tau_N}\sqrt{\frac{N}{2\pi}}\int_0^z e^{1-\zeta} \varphi(\zeta)^{N-1}d\zeta\\
    =&1-\frac{1}{\tau_N}\sqrt{\frac{N}{2\pi}}\left(ze^{1-z}\right)^N\int_0^1\frac{(1-t)^{N-1}}{1-\psi((1-t)ze^{1-z})}dt.
    \end{aligned}
\end{equation}
Define $f(z;t)= ( {1-\psi((1-t)ze^{1-z})})^{-1}$. Note that $f(z;\cdot)$ is infinitely differentiable in a neighborhood of $t=1$ since $\psi$ is analytic in a neighborhood of $ze^{1-z}$ for $|z|<1$. Direct differentiation  shows that
\begin{equation*}
    f(z;t)=\frac{1}{1-z}+r(z;t),
\end{equation*}
where 
\begin{equation}
    r(z;t)=\int_0^t\frac{-\psi((1-\tau)ze^{1-z})}{(1-\tau)(1-\psi((1-\tau)ze^{1-z}))^3}d\tau.
\end{equation}
By the continuity of $(\tau,z)\mapsto \psi((1-\tau)ze^{1-z})$ over $[0,1]\times A$, together with $\psi(ze^{1-z})=z$, there exists $t_A>0$ such that for all $0\leq \tau\leq t_A$ and all $z\in A$, $|1-\psi((1-\tau)ze^{1-z})|\leq \frac{1}{2}\tilde d_A$. Therefore, there exists $c_A >0$ depending on $A$ (only through $\tilde d_A$) such that $|r(z;t)|\leq c_At$.

A standard application of Laplace's method (see \cite[Section~19.2.4,~Theorem~1(a)]{zorich2016mathematical}) implies that there exists a constant $C>0$ depending only on $c_A$, and consequently only on $\tilde d_A$, such that 
\begin{equation}\label{expand4}
    \int_0^1(1-t)^{N-1}f(z;t)\, dt=\frac{1}{N(1-z)}(1+R(z;N)),
\end{equation}
where $|R(z;N)|\leq C N^{-1}$. Combining \eqref{expand3} and \eqref{expand4} and recalling the definition of $\tau_N$ in \eqref{stirling} completes the proof. 
\end{proof}
\begin{proof}[Proof of Lemma~\ref{l:23}]
Using \eqref{erf}, we obtain
\begin{equation*}
G(\sqrt{N} z , \sqrt{N} w)
= 
\frac{e^{-N(\Im(z)^2 + \Im(w)^2)}}
{\sqrt{2 N \pi} | \Im(z) \Im(w)| }\left( 1 + O\left(\frac{1}{N \min(|z|, |w|)^4} \right) \right).
\end{equation*}
Combining this estimate with Lemma~\ref{sNasymptotic} and $|(2\pi)^{-1} (w-z)|\le 1$, we get 
\begin{align}\label{Dexpand}
\big| D_N(\sqrt{N}z,\sqrt{N} w) \big|
\le & e^{-(N/2) \Re (z-w)^2}
\frac{e^{-N(\Im(z)^2 + \Im(w)^2)}}
{\sqrt{N } | \Im(z) \Im(w)| }\left( 1 + O\left(\frac{1}{N \min(|z|, |w|)^4} \right) \right)\notag \\ &\times
\left(1 +
\left|
\frac{e^{-2(1-zw)}}{\sqrt{2\pi N} (1- zw) }
e^{N(1-zw)}
(zw)^{N}
\left( 1 + O(N^{-1}) \right)
\right|
\right).
%
\end{align}
We observe that
\begin{equation}
-\frac{N}{2} \Re (z-w)^2
-N(\Im(z)^2 + \Im(w)^2)
\le - N\Im(z)\Im(w).\label{expbd1}
\end{equation}
We also note that 
\begin{align}
&|e^{-(N/2) \Re (z-w)^2}
e^{-N(\Im(z)^2 + \Im(w)^2)}
e^{N(1-zw)}
(zw)^{N}|\notag \\
&\le \exp\left(
- \frac{N}{2} ( |z|^2 - \ln |z|^2 -1 ) 
- \frac{N}{2} 
 ( |w|^2 - \ln |w|^2 -1 ) 
\right)\notag \\
&\le \exp\left(
- \frac{N}{8} ( |z|-1 )^2
- \frac{N}{8} 
 ( |w|  -1 )^2
\right).\label{expbd2}
\end{align}
Inserting \eqref{expbd1} and \eqref{expbd2} into \eqref{Dexpand} completes the proof of the bound on $D_N$. The proof for $I_N$ is similar, so we omit the details. 

For $S_N$, we have 
\begin{align*}
\big|
S_N(\sqrt{N} z, \sqrt{N} w)
\big| \le & e^{-(N/2) \Re (z-\overline{w} )^2}
\frac{e^{-N(\Im(z)^2 + \Im(w)^2)}}
{\sqrt{N } | \Im(z) \Im(w)| }\left( 1 + O\left(\frac{1}{N \min(|z|, |w|)^4} \right) \right)\notag \\ &\times
\left(1 +
\left|
\frac{e^{-2(1-z\overline{w})}}{\sqrt{2\pi N} \pi (1-z\overline{w}) }
e^{N(1-z\overline{w})}
(z\overline{w})^{N}
\big( 1 + O(N^{-1}) \big)
\right|
\right)
\end{align*}
We note that 
\begin{equation}
 e^{-(N/2) \Re (z-\overline{w} )^2}
 e^{-N(\Im(z)^2 + \Im(w)^2)} = e^{-(N/2) |z-w|^2}\le 1,
\end{equation}
and 
\begin{align}
& | e^{-(N/2) \Re (z-\overline{w} )^2}
 e^{-N(\Im(z)^2 + \Im(w)^2)} e^{N(1-z\overline{w})}
(z\overline{w})^{N} | \\
&\le e^{ (N/2) ( - |z-\overline{w}|^2 + 2 + 2 \Re( z \overline{w}) + 2 \ln |z \overline{w} | )} = e^{ - (N/2)( |z|^2 + |w|^2 - 2 - \ln |z|^2 - \ln |w|^2 )} \le 1,
\end{align}
where the last inequality follows from $|z|^2 - 1 - \ln |z|^2 \ge 0$ for $z\in A$. 
This completes the proof of the bound on $S_N$.
\end{proof}





\end{document}